\documentclass[psamsfonts,fceqn,leqno]{amsart}
\usepackage{mathrsfs,latexsym,amsfonts,amssymb,curves,epic}
\usepackage{color}

\newtheorem{theorem}{Theorem}[section]
\newtheorem{corollary}[theorem]{Corollary}
\newtheorem{proposition}[theorem]{Proposition}
\newtheorem{lemma}[theorem]{Lemma}

\theoremstyle{definition}

\def\N{{\mathbb N}}

\begin{document}

\title[The $k$-spaces property of free Abelian topological groups over non-metrizable La\v{s}nev spaces]
{The $k$-spaces property of free Abelian topological groups over non-metrizable La\v{s}nev spaces}

\author{Fucai Lin}
\address{(Fucai Lin): School of mathematics and statistics,
Minnan Normal University, Zhangzhou 363000, P. R. China}
\email{linfucai2008@aliyun.com; lfc19791001@163.com}

\author{Chuan Liu}
\address{(Chuan Liu): Department of Mathematics,
Ohio University Zanesville Campus, Zanesville, OH 43701, USA}
\email{liuc1@ohio.edu}

\thanks{The first author is supported by the NSFC (Nos. 11571158, 11201414, 11471153),
  the Natural Science Foundation of Fujian Province (No. 2012J05013) of China
  and Training Programme Foundation for Excellent Youth Researching Talents
  of Fujian's Universities (JA13190)}


\begin{abstract}
Given a Tychonoff space $X$, let $A(X)$ be the free Abelian topological group over $X$ in the sense
  of Markov. For every $n\in\mathbb{N}$, let $A_n(X)$ denote
  the subspace of $A(X)$ that consists of words of reduced length
  at most $n$ with respect to the free basis $X$. In this paper, we show that $A_4(X)$ is a $k$-space if and only if $A(X)$ is a $k$-space for the non-metrizable La\v{s}nev space $X$, which gives a complementary for one result of K. Yamada's.
  In addition, we also show that, under the assumption of $\flat=\omega_1$, the subspace $A_3(X)$ is a $k$-space if and only if $A(X)$ is a $k$-space for the non-metrizable La\v{s}nev space $X$. However, under the assumption of $\flat>\omega_1$, we provide a non-metrizable La\v{s}nev space $X$ such that $A_3(X)$ is a $k$-space but $A(X)$ is not a $k$-space.
\end{abstract}

\maketitle
\section{Introduction}
 In 1941, Markov \cite{MA} introduced the concept of the free Abelian topological group $A(X)$ over a Tychonoff
  space $X$. Since then, the free Abelian topological groups have been a
  source of various examples and also an interesting topic of study in the
  theory of topological groups, see \cite{AT2008}.
 For every $n\in\mathbb{N}$, let $A_n(X)$ denote the subspace of $A(X)$ that consists of words of reduced length
at most $n$ with respect to the free basis $X$. Recently, K. Yamada has showed the following two theorems:

\begin{theorem}\cite{Y1993}\label{metrizable k}
Let $X$ be a metrizable space. Then the following statements are equivalent.
\begin{enumerate}
\item $A_{n}(X)$ is a $k$-space for each $n\in\mathbb{N}$.
\item $A_4(X)$ is a $k$-space.
\item either $X$ is locally compact and the set $X'$ of all non-isolated points of $X$ is
separable, or $X'$ is compact.
\end{enumerate}
\end{theorem}

\begin{theorem}\cite{Y1996}\label{metrizable k1}
For the hedgehog space $J(\kappa)$ of spininess $\kappa\geq\aleph_{0}$, the space $A_{n}(J(\kappa))$ is a $k$-space for each $n\in\mathbb{N}$, but $A(J(\kappa))$
is not a $k$-space.
\end{theorem}

For a metrizable space $X$, it follow from Theorems~\ref{metrizable k} and ~\ref{metrizable k1} that the subspace $A_{4}(X)$ with the $k$-space property implies each $A_{n}(X)$ is a $k$-space, and that each $A_{n}(X)$ is a $k$-space does not imply $A(X)$ is a $k$-space. Therefore, it is natural to ask what the relations of subspaces $A_{4}(X)$, $A_{n}(X)$ and $A(X)$ with the $k$-space property for a non-metrizable space $X$. In this paper, we obtain an unexpected result and show that for any non-metrizable La\v{s}nev space $X$ (i.e. a non-metrizable, closed image of a  metric space), the subspace $A_4(X)$ is a $k$-space if and only if $A(X)$ is a $k$-space. Moreover, we also discuss the relations of subspaces $A_{3}(X)$ and $A_{4}(X)$ with the $k$-space property for a non-metrizable space $X$. Our results give a complementary for some results of K. Yamada's in literature.

\maketitle
\section{Preliminaries}
 In this section, we introduce the necessary notation and terminologies.
  First of all, let $\N$ denote the set of all positive
  integers. For a space $X$, we always denote the sets of all non-isolated points and isolated points of $X$ by $\mbox{NI}(X)$ and $\mbox{I}(X)$ respectively. All spaces are Tychonoff unless stated otherwise. For undefined
  notation and terminologies, the reader may refer to \cite{AT2008},
  \cite{E1989} and \cite{Gr1984}.

  \medskip
  Let $X$ be a topological space and $A \subseteq X$ be a subset of $X$.
  The \emph{closure} of $A$ in $X$ is denoted by $\overline{A}$ and the
  \emph{diagonal} of $X$ is denoted by $\Delta (X)$. Moreover, $A$ is called \emph{bounded} if every continuous real-valued
  function $f$ defined on $A$ is bounded.
  Recall that $X$ is said to have a \emph{$G_{\delta}$-diagonal} if $\Delta(X)$ is a $G_\delta$-set
 in $X \times X$. The space $X$ is called a
  \emph{$k$-space} provided that a subset $C\subseteq X$ is closed in $X$ if
  $C\cap K$ is closed in $K$ for each compact subset $K$ of $X$.
  If there exists
  a family of countably many compact subsets $\{K_{n}: n\in\mathbb{N}\}$ of
  $X$ such that each subset $F$ of $X$ is closed in $X$ provided that
  $F\cap K_{n}$ is closed in $K_{n}$ for each $n\in\mathbb{N}$, then $X$ is called a \emph{$k_{\omega}$-space}. Note that every $k_{\omega}$-space is a $k$-space. A subset $P$ of $X$ is called a
  \emph{sequential neighborhood} of $x \in X$, if each
  sequence converging to $x$ is eventually in $P$. A subset $U$ of
  $X$ is called \emph{sequentially open} if $U$ is a sequential neighborhood of
  each of its points. The space $X$ is called a \emph{sequential  space} if each
  sequentially open subset of $X$ is open. The space $X$ is said to be {\it Fr\'{e}chet-Urysohn} if, for
each $x\in \overline{A}\subset X$, there exists a sequence
$\{x_{n}\}$ such that $\{x_{n}\}$ converges to $x$ and $\{x_{n}:
n\in\mathbb{N}\}\subset A$. Let $\kappa$ be an infinite cardinal.
  For each $\alpha\in\kappa$, let $T_{\alpha}$ be a sequence converging to
  $x_{\alpha}\not\in T_{\alpha}$. Let $T:=\bigoplus_{\alpha\in\kappa}(T_{\alpha}\cup\{x_{\alpha}\})$ be the topological sum of $\{T_{\alpha}
  \cup \{x_{\alpha}\}: \alpha\in\kappa\}$. Then
  $S_{\kappa} :=\{x\}  \cup \bigcup_{\alpha\in\kappa}T_{\alpha}$
  is the quotient space obtained from $T$ by
  identifying all the points $x_{\alpha}\in T$ to the point $x$.

  \medskip
  Let $\mathscr P$ be a family of subsets of $X$. Then, $\mathscr P$ is called a {\it $k$-network}
  \cite{O1971} if for every compact subset $K$ of $X$ and an arbitrary open set
  $U$ containing $K$ in $X$ there is a finite subfamily $\mathscr {P}^{\prime}\subseteq
  \mathscr {P}$ such that $K\subseteq \bigcup\mathscr {P}^{\prime}\subseteq U$. Recall that a space $X$ is an \emph{$\aleph$-space} (resp.
  {\it $\aleph_{0}$-space}) if
  $X$ has a $\sigma$-locally finite (resp. countable) $k$-network. Recall that a topological space $X$ is said to be
  \emph{La\v{s}nev} if it is the closed image of some metric space. The following two well known facts about the La\v{s}nev spaces shall be used in this paper.

 \medskip
 {\bf Fact 1:} A La\v{s}nev space is metrizable if it contains no closed copy of $S_\omega$.

 \medskip
 {\bf Fact 2:} A La\v{s}nev space is an $\aleph$-space if it contains no closed copy of $S_{\omega_1}$.

  \medskip
  Let $X$ be a non-empty Tychonoff space. Throughout this paper, $-X: =\{-x: x\in X\}$, which is just a copy of
  $X$. For every $n\in\mathbb{N}$, $A_{n}(X)$ denotes the
  subspace of $A(X)$ that consists of all words of reduced length at
  most $n$ with respect to the free basis $X$. Let $0$ be the neutral element of $A(X)$ (i.e., the empty
  word). For every $n\in\N$ and an element
  $x_{1}+x_{2}+\cdots +x_{n}$ is also called a {\it form} for $(x_{1}, x_{2},
  \cdots, x_{n})\in(X\bigoplus -X\bigoplus\{0\})^{n}$. This word $g$ is
  called {\it reduced} if it does not contains $0$ or any pair of
  consecutive symbol of the form $x-x$. It follows
  that if the word $g$ is reduced and non-empty, then it is different
  from the neutral element $0$ of $A(X)$. In particular, each element
  $g\in A(X)$ distinct from the neutral element can be uniquely written
  in the form $g=r_{1}x_{1}+r_{2}x_{2}+\cdots +r_{n}x_{n}$, where
  $n\geq 1$, $r_{i}\in\mathbb{Z}\setminus\{0\}$, $x_{i}\in X$, and
  $x_{i}\neq x_{j}$ for $i\neq j$, and the {\it support}
  of $g=r_{1}x_{1}+ r_{2}x_{2}+\cdots +r_{n}x_{n}$ is defined as
  $\mbox{supp}(g) :=\{x_{1}, \cdots, x_{n}\}$. Given a subset $K$ of
  $A(X)$, we put $\mbox{supp}(K):=\bigcup_{g\in K}\mbox{supp}(g)$. For every $n\in\mathbb{N}$, let $$i_n: (X\bigoplus -X
  \bigoplus\{0\})^{n} \to A_n(X)$$ be the natural mapping defined by
  $$i_n(x_1, x_2, ... x_n)= x_1+x_2+...+x_n$$
  for each $(x_1, x_2, ... x_n) \in (X\bigoplus -X
  \bigoplus\{0\})^{n}$.

\maketitle
\section{main results}
First, we give a characterization of a non-metrizable La\v{s}nev space $X$ such that $A_{4}(X)$ is a $k$-space, see Theorem~\ref{k-space-characterization}. In order to obtain this result, we first prove some propositions and lemmas.

\begin{proposition}\label{t1}
If $A(X)$ is a sequential space, then either $X$ is a discrete space or $A(X)$ contains a closed copy of $S_\omega$.
\end{proposition}

\begin{proof}
Assume on the contrary that $X$ is not a discrete space. Since $A(X)$ is sequential and $X$ is closed in $A(X)$, the space $X$ is sequential, hence there are a point $x\in X$ and a non-trivial sequence $\{x_n: n\in \mathbb{N}\}\subset X\setminus\{x\}$ such that $\{x_n: n\in \mathbb{N}\}$ converges to $x$. For each $k\in\mathbb{N}$, let $L_k: =\{kx_n-kx: n\in \mathbb{N}\}$; then $L_k\subset A_{2k}(X)\setminus A_{2k-1}(X)$ and $L_k\to 0$ as $k\rightarrow\infty$. Let $L: =\bigcup\{L_k: k\in \mathbb{N}\}\cup\{0\}$. Next it suffices to show that $L$ is a closed copy of $S_\omega$.

Suppose $L$ is not closed in $A(X)$ or contains no copy of $S_\omega$. Then, since $A(X)$ is sequential, there is a sequence $\{y_n: n\in\mathbb{N}\}\subset L$ such that $y_n\to y\notin L$ and $\{y_n\}$ meets infinitely many $L_k$'s. This fact implies that there is a subsequence $K\subset \{y_n: n\in\mathbb{N}\}$ such that $K\cap A_n(X)$ is finite, hence $K$ is discrete by \cite[Corollary 7.4.3]{AT2008}, which is a contradiction. Therefore, $L$ is a closed copy of $S_\omega$.
\end{proof}

The following lemma was proved in \cite {LL2010}.

\begin{lemma}\label{l1}
\cite {LL2010} Suppose $X$ is a sequential topological group, then either $X$ contains no closed copy of $S_\omega$ or every closed first-countable subspace of $X$ is locally countably compact.
\end{lemma}

By Lemma~\ref{l1}, we can show the following proposition.

\begin{proposition}\label{locally-countably-compact}
Let $X$ be a first-countable space. If $A(X)$ is sequential, then $X$ is locally countably compact.
\end{proposition}

\begin{proof}
If $X$ is discrete, then it is obvious that $X$ is locally countably compact. Assume that $X$ is not discrete, then it follows from Proposition~\ref{t1} that $A(X)$ contains a closed copy of $S_\omega$. Therefore, by Lemma ~\ref{l1}, every first-countable closed subspace of $A(X)$ is locally countably compact. Since $X$ is first-countable and closed in $A(X)$, the space $X$ is locally countably compact.
\end{proof}

By Proposition~\ref{locally-countably-compact} and the following lemma, we have Corollary~\ref{locally}, which was proved in \cite{AOP1989}. The concept of stratifiable space can be seen in \cite{B1966}.

\begin{lemma}\cite{S1993}\label{stratifiable k}
Let $X$ be a stratifiable space. Then $A(X)$ is also a stratifiable space.
\end{lemma}

\begin{corollary}\cite{AOP1989}\label{locally} If $X$ is a metrizable space and $A(X)$ is a $k$-space, then $X$ is locally compact.
\end{corollary}

\begin{proof}
Since $X$ is metrizable, it follows from Lemma~\ref{stratifiable k} that $A(X)$ is stratifiable, hence it is a sequential space. By Proposition~\ref{locally-countably-compact}, the space $X$ is locally countably compact, then it is locally compact since every countably compact subset of a metrizable space is compact.
\end{proof}

Next we recall two spaces $M_{1}$ and $M_{3}$, which were introduced in \cite{Y1993}.

Let $M_1: =\{x\}\cup (\bigcup\{X_i: i\in \mathbb{N}\})$ such that each $X_i$ is an infinite, countable, discrete and open subspace of $M_1$ and the family $\{V_k=\{x\}\cup \bigcup \{X_i: i\geq k\}: k\in \mathbb{N}\}$ is a neighborhood base of the point $x$ in $M_1$. Let
$M_3: =\bigoplus\{C_\alpha: \alpha<\omega_1\}$, where, for each $\alpha<\omega_1$, the set $C_\alpha: =\{x(n, \alpha): n\in \mathbb{N}\}\cup\{x_\alpha\}$ with $x(n, \alpha)\to x_\alpha$ as $n\rightarrow\infty$.

\begin{lemma}\label{l-sequential-a}
Let $m_{0}$ and $n_{0}$ be two natural numbers. If $A_{m_{0}}(X)$ contains a closed copy of $S_\omega$ and $A_{n_{0}}(X)$ contains a closed copy of the space $M_1$, then $A_{m_{0}+n_{0}}(X)$ is not a sequential space.
\end{lemma}

\begin{proof}
Assume that $A_{m_{0}+n_{0}}(X)$ is a sequential space.
Let $\{x_{0}\}\cup\{x(n, m): n, m\in\mathbb{N}\}$ be a closed copy of $S_\omega$ in $A_{m_{0}}(X)$, where $x(n, m)\rightarrow x_{0}$ as $m\rightarrow\infty$ for each $n\in\mathbb{N}$; let $\{y_{0}\}\cup\{y(n, m): n, m\in\mathbb{N}\}$ be a closed copy of space $M_1$ in $A_{n_{0}}(X)$, where the set $\{y(n, m): m\in\mathbb{N}\}$ is discrete and open in $A_{n_{0}}(X)$ for each $n\in\mathbb{N}$. Put $$H: =\{x(n, m)-y(n, m): n, m\in\mathbb{N}\}.$$

We claim that $H$ contains no non-trivial convergent sequence. Assume on the contrary that there exists a non-trivial sequence $\{x(n(k), m(k))-y(n(k), m(k)): k\in\mathbb{N}\}$ converging to $z\in A_{m_{0}+n_{0}}(X)$. Without loss of generality, we may assume that $x(n(k), m(k))-y(n(k), m(k))\neq x(n(l), m(l))-y(n(l), m(l))$ if $k\neq l.$ If there exists some $p\in\mathbb{N}$ such that $\{k\in\mathbb{N}: n(k)=p\}$ is an infinite set. By the assumption of the sequence, there exists a subsequence $\{x(n(l_{k}), m(l_{k}))-y(n(l_{k}), m(l_{k})): k\in\mathbb{N}\}$ of $\{x(n(k), m(k))-y(n(k), m(k)): k\in\mathbb{N}\}$ such that $n(l_{k})=p$ for each $k\in\mathbb{N}$ and $\{m(l_{k}): k\in\mathbb{N}\}$ is an infinite set. Then we have $-y(n(l_{k}), m(l_{k}))=x(n(l_{k}), m(l_{k}))-y(n(l_{k}), m(l_{k}))-x(n(l_{k}), m(l_{k}))\rightarrow z-x_{0}$ as $k\rightarrow\infty$, which is a contradiction. Therefore, there does not exist $p\in\mathbb{N}$ such that $\{k\in\mathbb{N}: n(k)=p\}$ is an infinite set. Similar to the above proof, ones can see that there does not exist $p\in\mathbb{N}$ such that $\{k\in\mathbb{N}: m(k)=p\}$ is an infinite set. Therefore, we have $x(n(k), m(k))=x(n(k), m(k))-y(n(k), m(k))+y(n(k), m(k))\rightarrow z+y_{0}$ as $k\rightarrow\infty$, which is a contradiction. Thus $H$ contains no non-trivial convergent sequence.

By the above claim,  it follows that $H$ is sequentially closed in $A_{m_{0}+n_{0}}(X)$, then $H$ is closed in $A_{m_{0}+n_{0}}(X)$ since $A_{m_{0}+n_{0}}(X)$ is a sequential space.
 However, it is easy to see that $x_0-y_0\in \overline{H}$ and $x_0-y_0\not\in H$, which is a contradiction. Thus $A_{m_{0}+n_{0}}(X)$ is not a sequential space.
\end{proof}

The following lemma was proved in \cite{Y1993}, which plays an important role in the proof of our main theorem.

\begin{lemma}\label{s-omega-1}
\cite[Theorem 3.6]{Y1993} The subspace $A_4(M_3)$ is not a $k$-space.
\end{lemma}

\begin{lemma}\label{aleph-space}
Let $X$ be a La\v{s}nev space. If $A_{2}(X)$ is a $k$-space, then $X$ is an $\aleph$-space.
\end{lemma}

\begin{proof}
By Fact 2, it suffices to show that $X$ contains no any closed copy of $S_{\omega_{1}}$. Assume that $X$ contains a closed copy of $S_{\omega_{1}}$. Then it follows from \cite[Corollary 2.2]{Y1997} that $A_{2}(X)$ is not a sequential space. However, by Lemma~\ref{stratifiable k}, the subspace $A_{2}(X)$ is a stratifiable space, hence it has a $G_{\delta}$-diagonal, and then $A_{2}(X)$ is a sequential space since a $k$-space with a $G_{\delta}$-diagonal is sequential \cite{Gr1984}.
\end{proof}

Recall that a space is called $\omega_1$-{\it compac} if every uncountable subset of $X$ has a cluster point.
Now we can show one of our main theorems.

\begin{theorem}\label{k-space-characterization}
Let $X$ be a non-metrizable La\v{s}nev space. Then the following statements are equivalent.
\begin{enumerate}
\item $A(X)$ is a sequential space.
\item $A(X)$ is a $k$-space.
\item $A_n(X)$ is a $k$-space for each $n\in\mathbb{N}$.
\item $A_4(X)$ is a $k$-space.
\item $X$ is a topological sum of a space with a countable $k$-network consisting of compact subsets and a discrete space.
\end{enumerate}
\end{theorem}

\begin{proof}
The implications of (1) $\Rightarrow$ (2), (2) $\Rightarrow$ (3) and (3) $\Rightarrow$ (4) are trivial. The implication of (2) $\Rightarrow$ (1) follows from Lemma~\ref{stratifiable k}. It suffices to show (4) $\Rightarrow$ (5) and (5) $\Rightarrow$ (2).

(4) $\Rightarrow$ (5). First, we show the following claim.

{\bf Claim 1}: The subspace $\mbox{NI}(X)$ is $\omega_1$-compact in $X$.

Assume on the contrary that there exists a closed, discrete and uncountable subset $\{x_\alpha: \alpha<\omega_1\}$ in $\mbox{NI}(X)$. Since $X$ is paracompact and $\mbox{NI}(X)$ is closed in $X$, there is an uncountable and discrete collection of open subsets $\{U_\alpha: \alpha<\omega_1\}$ in $X$ such that $x_\alpha\in U_\alpha$ for each $\alpha<\omega_1$. Since $X$ is Fr\'echet-Urysohn, for each $\alpha<\omega_1$, let $\{x(n, \alpha): n\in \mathbb{N}\}$ be a non-trivial sequence converging to $x_\alpha$ in $X$. For each $\alpha<\omega_1$, let $$C_\alpha: =\{x(n, \alpha): n\in \mathbb{N}\}\cup \{x_\alpha\}$$ and put $$M_3: =\bigcup\{C_\alpha: \alpha<\omega_1\}.$$ Since $M_3$ is a closed subset of $X$ and $X$ is a La\v{s}nev space, it follows from \cite{U1991} that the subspace $A_4(M_3)$ is homeomorphic to a closed subset of $A_4(X)$, thus $A_4(M_3)$ is sequential. However, by Lemma ~\ref{s-omega-1},the subspace  $A_4(M_3)$ is not sequential, which is a contradiction. Therefore, Claim 1 holds.

Since $X$ a non-metrizable La\v{s}nev space, the space $X$ must contain a closed copy of $S_\omega$ by Fact 1. Moreover, since $A_{4}(X)$ is a sequential space, the subspace $A_2(X)$ is also a sequential space. Then, by Lemma~\ref{l-sequential-a}, the space $X$ contains no closed copy of the space $M_1$. In addition, $X$ is an $\aleph$-space by Lemma~\ref{aleph-space}. Hence there exists a $\sigma$-locally finite $k$-network $\mathcal{P}$ in $X$ such that $\overline{P}$ is compact for each $P\in \mathcal{P}$. Since every compact subset of $X$ is metrizable in a La\v{s}nev space \cite{Gr1984}, the family $\mathcal{P}$ is a $\sigma$-locally-finite $k$-network consisting of separable metric subsets of $X$. Let $$\mathcal{P}_1: =\{\overline{P}: P\in \mathcal{P}, P\cap \mbox{NI}(X)\neq \emptyset\}.$$ By the $\omega_1$-compactness of $\mbox{NI}(X)$, the family $\{P\in \mathcal{P}: P\cap \mbox{NI}(X)\neq \emptyset\}$ is countable since at most countably many elements of an arbitrary locally-finite family intersect an $\omega_1$-compact subset. Therefore, we have $|\mathcal{P}_1|<\omega_1$. Note that $|\overline{P}\cap \mbox{I}(X)|\leq \omega$, hence $$\mathcal{P}_1\cup\{\{y\}: y\in \overline{P}\cap \mbox{I}(X), P\in \mathcal{P}_1\}$$ is a countable $k$-network of $X_1: =\bigcup \mathcal{P}_1$. It is easy to prove that $X_1$ is sequentially open, thus it is open in $X$. Let $X_2: =X\setminus X_1$. Then $X_2$ is an open and closed discrete subset of $X$, thus $X=X_1\bigoplus X_2$.

(5) $\Rightarrow$ (2). Let $X: =X_1\bigoplus X_2$, where $X_1$ is a $k$-space with a countable $k$-network consisting of compact subsets in $X_1$ and $X_2$ is a discrete space. It follows from \cite{T1974} that $A(X)\cong A(X_1)\times A(X_2)$. Since $X_1$ is a $k$-space with a countable $k$-network consisting of compact subsets, the space $X_1$ is a $k_\omega$-space, hence it follows from \cite[Theorem 7.4.1]{AT2008} that $A(X_1)$ is a $k_\omega$-space, then $A(X)$ is a $k$-space since $A(X_2)$ is discrete.
\end{proof}

It is natural to consider that if it can be replaced ``$A_4(X)$'' with ``$A_3(X)$ or $A_2(X)$'' in Theorem~\ref{k-space-characterization}. Next we shall give a complete answer to this question. First, the following proposition shows that it can not be replaced ``$A_4(X)$'' with ``$A_2(X)$'' in Theorem~\ref{k-space-characterization}.

\begin{proposition}
Let the space $X: =S_\omega\bigoplus M_3$. Then $A_2(X)$ is a sequential space, but $A(X)$ is not a sequential space.
\end{proposition}

\begin{proof}
First of all, it is not difficult to verify that $(X\bigoplus -X\bigoplus\{0\})^{2}$ is a sequential space. Then, by \cite[Proposition 4.8]{AOP1989}, the mapping $i_2: (X\bigoplus -X\bigoplus\{0\})^{2}\to A_2(X)$ is closed, which shows that $A_2(X)$ is sequential. Assume that $A(X)$ is sequential, then $A(M_3)$ is sequential. Hence $\mbox{NI}(M_3)$ is separable \cite[Theorem 2.11]{AOP1989}. This is a contradiction, thus $A(X)$ is not sequential.
\end{proof}

In order to prove that it can not be replaced ``$A_4(X)$'' with ``$A_3(X)$'' in Theorem~\ref{k-space-characterization}, we must recall some concepts.

Consider $^\omega\omega$, the collection of all functions from $\omega$ to $\omega$. We define a quasi-order $\leq^*$ on $^\omega\omega$ by specifying that
if $f, g\in ^\omega\omega$, then $f\leq^* g$ if $f(n)\leq g(n)$ for all but finitely many $n\in \omega$. A subset $\mathscr{F}$ of $^\omega\omega$ is {\it bounded} if there is a $g\in$$^\omega\omega$ such that $f\leq^* g$ for all $f\in\mathscr{F}$, and is {\it unbounded} otherwise. We denote by $\flat$ the smallest cardinality of an unbounded family in $^\omega\omega$. It is well known that $\omega <\flat\leq \mathrm{c}$, where $\mathrm{c}$ denotes the cardinality of the continuum.

Let $\mathcal{U}_X$ be the universal uniformity on a space $X$ and put $\mathcal{P}: =\{P\subset \mathcal{U}_X: |P|\leq \omega\}$. For each $P=\{U_i\}_{i\in\mathbb{N}}\in \mathcal{P}^{\omega}$, let $$W(P): =\{x_1-y_1+x_2-y_2+ ... +x_k-y_k: (x_i, y_i)\in U_i, i\leq k, k\in \mathbb{N}\}$$
and $$\mathcal{W}: =\{W(P): P\in \mathcal{P}^{\omega}\}.$$

In \cite{Y1993}, K. Yamada showed the following important result, which gives a neighborhood base of 0 in $A(X)$.

\begin{theorem}\label{yamada-neighborhood}
\cite[Theorem 2.3]{Y1993} The family $\mathcal{W}$ is a neighborhood base of 0 in $A(X)$.
\end{theorem}

The following theorem shows that we can replace ``$A_4(X)$'' with ``$A_3(X)$'' in Theorem~\ref{k-space-characterization} under the assumption of $\flat=\omega_1$. Note that the following proofs contain some ideas in \cite{Gr1980}.

\begin{theorem}
Assume $\flat=\omega_1$. For a non-metrizable La\v{s}nev space $X$, the subspace $A_3(X)$ is a sequential space if and only if $A(X)$ is a sequential space.
\end{theorem}

\begin{proof}
Clearly, it suffices to show the necessity. In order to show the sequentiality of $A(X)$, it suffices to show that the subspace $\mbox{NI}(X)$ is $\omega_1$-compact in $X$ by the proof of (4) $\Rightarrow$ (5) of Theorem ~\ref{k-space-characterization}. Next we shall show that the subspace $\mbox{NI}(X)$ is $\omega_1$-compact in $X$.

Suppose that $\flat=\omega_1$ and $A_3(X)$ is a sequential space. Then there exists a collection $\{f_\alpha\in ^\omega\omega: \alpha<\omega_1\}$ such that if $f\in ^{\omega}\omega$, then there exists $\alpha<\omega_1$ with $f_\alpha(n)>f(n)$ for infinitely may $n\in \omega$. Since $X$ is a non-metrizable La\v{s}nev space, the space $X$ contains a closed copy of $S_\omega$ by Fact 1. We rewrite the copy of $S_\omega$ as $Y: =\{y\}\cup\{y(n, m): m, n, \in \omega\}$, where $y(n, m)\to y$ as $m\to \infty$ for each $n\in\mathbb{N}$.

Assume on the contrary that the subspace $\mbox{NI}(X)$ is not $\omega_1$-compact in $X$. Then, by viewing the proof of $(3)\Rightarrow (4)$ in Theorem~\ref{k-space-characterization}, we can see that $X$ contains a closed copy of $M_3: =\bigoplus\{C_\alpha: \alpha<\omega_1\}$, where, for each $\alpha\in\omega_1$, the set $C_\alpha: =\{x(n, \alpha): n\in \omega\}\cup \{x_\alpha\}$ and $x(n, \alpha)\to x_\alpha$ as $n\rightarrow\infty$. Moreover, without loss of generality, we may assume that $Y \cap M_3=\emptyset$. Let $Z: =Y\bigoplus M_3$. Then we can define a uniform base $\mathcal{U}$ of the universal uniformity on $Z$ as follows.
For each $\alpha<\omega_1$ and $k\in\omega$, let $$V_{k, \alpha}: =\{x(m, \alpha): m\geq k\}\cup\{x_\alpha\}$$ and $$U_{k, \alpha}: =(V_{k, \alpha}\times V_{k, \alpha})\cup\Delta_\alpha,$$ where $\Delta_\alpha$ is the diagonal of $C_\alpha\times C_\alpha$. For each $f\in ^\omega\omega$ and $g\in ^{\omega_1}\omega$, let $$U(g, f): =\bigcup\{U_{g(\alpha), \alpha}: \alpha<\omega_1\}\cup (V_f\times V_f)\cup \Delta_Z,$$ where $V_f: =\{y\}\cup \{y(n, m): m\geq f(n), n\in \omega\}$ and $\Delta_Z$ is the diagonal of $Z\times Z$. Put $$\mathscr{U}: =\{U(g, f): g\in ^{\omega_1}\omega, f\in ^\omega\omega\}.$$
Then, the family $\mathscr{U}$ is a uniform base of the universal uniformity on the space $Z$. Put $\mathcal{W}: =\{W(P): P\in\mathscr{U}^{\omega}\}$. Then it follows from Theorem~\ref{yamada-neighborhood} that $\mathcal{W}$ is a neighborhood base of 0 in $A(Z)$.

For each $\alpha <\omega_1$, let $$H_\alpha: =\{y(n, m)+x(n, \alpha)-x_\alpha: m\leq f_\alpha(n), n\in \omega\}.$$ Let $H: =\cup_{\alpha<\omega_1}H_\alpha$. Then $H\subset A_3(Z)$. Then we have the following claim.

\smallskip
{\bf Claim 1:} The point $y\in \overline{H}\setminus H$ in $A_3(X)$.

\smallskip
{\it Proof of Claim 1.} Since $X$ is a La\v{s}nev space and $Z$ is closed in $X$, the group $A(Z)$ is topologically homeomorphic to a closed subgroup of $A(X)$. Therefore, it suffices to show that $y\in \overline{H}\setminus H$ in $A_3(Z)$. Obviously, the family $\{(y+U)\cap A_3(Z): U\in \mathcal{W}\}$ is a neighborhood base of $y$ in $A_3(Z)$.
Next we shall prove $(y+U)\cap A_3(Z)\cap H\neq \emptyset$ for each $U\in \mathcal{W}$, which implies $y\in \overline{H}\setminus H$ in $A_3(Z)$. Fix an $U\in\mathcal{W}$.  Then there exist a sequence $\{h_{i}\}_{i\in\omega}$ in $^{\omega_1}\omega$ and a sequence $\{g_{i}\}_{i\in\omega}$ in $^{\omega}\omega$ such that $$U=\{x_1-y_1+x_2-y_2+...+x_n-y_n: (x_i, y_i)\in U(h_i, g_i), i\leq n, n\in \omega\}.$$ Let $$B: =\{x'-y+x''-y'': (x', y)\in V_{g_1}\times V_{g_1}, (x'', y'')\in U_{h_{1}(\alpha), \alpha}, \alpha<\omega_1\}.$$ Then $B\subset U$. By the assumption, there exists $\alpha< \omega_1$ such that $f_\alpha (k)\geq g_1(k)$ for infinitely many $k$. Pick a $k'>h_{1}(\alpha)$ such that $(y(k', f_\alpha(k')), y)\in V_{g_1}\times V_{g_1}$. Then $$y(k', f_\alpha(k'))-y+x(k', \alpha)-x_\alpha\in U,$$ hence
\begin{eqnarray}
y+y(k', f_\alpha(k'))-y+x(k', \alpha)-x_\alpha&=&y(k', f_\alpha(k'))+x(k', \alpha)-x_\alpha\nonumber\\
&\in&((y+U)\cap A_3(Z))\cap H_\alpha.\nonumber
\end{eqnarray}
The proof of Claim 1 is completed.

\medskip
Since $A_3(X)$ is sequential and $H$ is not closed in $A_{3}(X)$ by Claim 1, there is a sequence $L\subset H$ such that $L\to z$ for some $z\in A_3(X)\setminus H$. Then the set $\mbox{supp}(L\cup\{z\})$ is bounded by \cite{AOP1989}, then the closure of $\mbox{supp}(L\cup\{z\})$ is compact since $X$ is a La\v{s}nev space. If $L$ meets infinitely many $H_\alpha$'s, then the set $\mbox{supp}(L\cup\{z\})$ contains infinitely many $x_\alpha$'s. This is a contradiction since $\{x_\alpha: \alpha<\omega_1\}$ is  discrete in $X$. If $L\cup\{z\}$ is contained in the union of finitely many $H_\alpha$'s, then the set $\mbox{supp}(L\cup\{z\})$ contains an infinite subset $\{y(n_i, m_i): i\in \mathbb{N}\}$ of $\{y(n, m): n, m\in \mathbb{N}\}$ such that $n_i\neq n_j$ if $i\neq j$, which is a contradiction. Therefore, the subspace $\mbox{NI}(X)$ is $\omega_1$-compact in $X$.
\end{proof}

Finally, we shall give an example to show that it can not be replaced ``$A_4(X)$'' with ``$A_3(X)$'' in Theorem~\ref{k-space-characterization} under the assumption of $\flat>\omega_1$. First, we give a technical lemma.

Let $C\subset M_3$ be a convergent sequence with the limit point. Since the subspaces $C\bigoplus S_\omega, M_3$ and $S_\omega$ are all $P^*$-embedded in the space $S_\omega\bigoplus M_3$ respectively, it follows from \cite[Problem 7.7.B.]{AT2008} that the following three subgroups $$A(C\bigoplus S_\omega, S_\omega\bigoplus M_3), A(M_3, S_\omega\bigoplus M_3)\ \mbox{and}\ A(S_\omega, S_\omega\bigoplus M_3)$$ in $A(S_\omega\bigoplus M_3)$ are topologically isomorphic to the groups $A(C\bigoplus S_\omega)$, $A(M_3)$ and $A(S_\omega)$ respectively. Then, by \cite[Corollary 7.4.9]{AT2008} and \cite[Theorem 4.9]{Y1993}, $A(C\bigoplus S_\omega)$, $A(M_3)$ and $A(S_\omega)$ are all sequential, hence we have the following lemma.

\begin{lemma}\label{l-sequential}
The following three subspaces $$A_3(C\bigoplus S_\omega, S_\omega\bigoplus M_3), A_3(M_3, S_\omega\bigoplus M_3)\ \mbox{and}\ A_3(S_\omega, S_\omega\bigoplus M_3)$$ are all sequential and closed in $A(S_\omega\bigoplus M_3)$.
\end{lemma}

Let $S(A)$ denote the set of all limit points of a subset $A$ in a space $X$. In order to show the last theorem, we need the following technical lemma.

\begin{lemma}\label{l-4}
Let $B$ be a sequentially closed subset of $A_{3}(S_\omega\bigoplus M_3)$.

\smallskip
(a) If $$B\cap \{y+x'-y': (x', y')\in V_f\times V_f\}\neq \emptyset$$ for any $f\in ^\omega\omega$, then $y\in B$.

\smallskip
(b) If $$B\cap \{y+x'-y': (x', y')\in V_{h(\alpha), \alpha}\times V_{h(\alpha), \alpha}, \alpha<\omega_1\}\neq \emptyset$$  for any $h\in ^{\omega_1}\omega$, then $y\in B$.

\smallskip
(c) If $$B\cap \{y+x'-y'+x''-y'': (x', y')\in V_f\times V_f, (x'', y'')\in V_g\times V_g\}\neq \emptyset$$ for any $f, g\in ^\omega\omega$, then $y\in B$.
\end{lemma}

\begin{proof}
($a$) Let $B_1: =B\cap A_3(S_\omega, S_\omega\bigoplus M_3)$. Then $B_1$ is sequentially closed in the subspace $A_3(S_\omega, S_\omega\bigoplus M_3)$, hence it is closed by Lemma ~\ref{l-sequential}. Since $$B_1\cap \{y+x'-y': (x', y')\in V_f\times V_f\}=B\cap \{y+x'-y': (x', y')\in V_f\times V_f\}\neq \emptyset$$ for any $f\in ^\omega\omega$, we have $y=y+y-y\in \overline{B_1}=B_1\subset B$

\smallskip
($b$) Since $B\cap \{y+x'-y': (x', y')\in V_{h(\alpha), \alpha}\times V_{h(\alpha), \alpha}, \alpha<\omega_1\}\neq \emptyset$ for any $h\in ^{\omega_1}\omega$, it is easy to see that there exists $\beta$ such that $$B\cap \{y+x'-y': (x', y')\in V_{h(\beta), \beta}\times V_{h(\beta), \beta}\}\neq \emptyset$$ for any $h\in ^{\omega_1}\omega$. Then it easily obtain that $y=y+x_\beta-x_\beta\in S(B)=B$.

\smallskip
($c$) Let $B_1: = B\cap A_3(S_\omega, S_\omega\bigoplus M_3)$. By \cite[Theorem 7.4.5]{AT2008} and Lemma ~\ref{l-sequential}, the subspace $A_3(S_\omega, S_\omega\bigoplus M_3)$ is closed and sequential, which implies that $B_1$ is closed in $A_3(S_\omega, S_\omega\bigoplus M_3)$. Let $$A: =B_{1}\cap \{y+x'-y'+x''-y'': (x', y')\in V_f\times V_f, (x'', y'')\in V_g\times V_g, f, g \in ^\omega\omega\}.$$ Then $y\in \overline{A}$, and $A\subset B_1$. It implies that $y\in B_1\subset B$.
\end{proof}

\begin{theorem}
Assume $\flat>\omega_1$. There exists a non-metrizable La\v{s}nev space $X$ such that $A_3(X)$ is a sequential space but $A(X)$ is not a sequential space.
\end{theorem}

\begin{proof}
Let $X=S_\omega\bigoplus M_3$; then $X$ is a non-metrizable La\v{s}nev space. However, $A(X)$ is not a sequential space. Indeed, assume on the contrary that $A(X)$ is a sequential space. Since $A(M_3)$ is topologically homeomorphic to a closed subset of $A(X)$, the group $A(M_3)$ is sequential, then it follows from  \cite[Theorem 2.11]{AOP1989} that $\mbox{NI}(M_3)$ is separable, which is a contradiction since $\mbox{NI}(M_3)$ is a uncountable closed discrete subset in $M_3$.

Next, we shall prove that $A_3(X)$ is a sequential space. Assume on the contrary that $A_3(X)$ is not a sequential space. Then there exists a subset $H$ in $A_3(X)$ such that $H$ is a sequentially closed subset in $A_3(X)$, but not closed in $A_3(X)$. Then there exists a point $x_{0}$ belonging to $A_3(X)$ such that $x_{0}\in \overline{H}\setminus H$ in $A_3(X)$.
Let $H_1: =A_2(X)\cap H$ and $H_2: =(A_3(X)\setminus A_2(X))\cap H$. Since $H_1$ is sequentially closed and $A_2(X)$ is sequential, the set $H_1$ is closed. If $x_0\in \overline{H_1}$, then $x_0\in \overline{H_1}=H_1\subset H$, this is a contradiction. Hence $x_{0}\in \overline{H_2}$. Then since $A_3(X)$ is paracompact (in fact, $A_3(X)$ is stratifiable), there exist open subsets $W_1, W_2$ in $A_3(X)$ such that $x_0\in W_1, H_1\subset W_2$ and $\overline{W_1}\cap \overline{W_2}=\emptyset$. Let $H': =\overline{W_2}\cap H\subset A_3(X)\setminus A_2(X)$. Then the point $x_0$ belongs to $\overline{H'}$,  the set $H'$ is sequentially closed and $H'$ does not have any limit point in $X\cup (-X)$. Without loss of generality, we may assume $H=H'$.

Since the length of each element of $H$ is 3, the length of $x_{0}$ is 1 or 3.
However, the length of $x_{0}$ can not be 3. Assume on the contrary that $x_{0}\in A_3(X)\setminus A_2(X)$; then let $V$ be an open neighborhood of $x_{0}$ in $A_3(X)\setminus A_2(X)$ such that $\overline{V}\subset A_3(X)\setminus A_2(X)$. Then $\overline{V}\cap H$ is sequentially closed in the sequential subspace $A_3(X)\setminus A_2(X)$, hence $\overline{V}\cap H$ is closed and $x_{0}\in H$, this is a contradiction. Therefore, the length of $x_{0}$ is 1. Without loss of generality, we may assume that $x_{0}\in X$. In order to obtain a contradiction, we shall show that $x_{0}=y$, but $y\not\in \overline{H}$. We divide the proof in two
claims.

\smallskip
{\bf Claim 2:} We have $x_{0}=y$.
\smallskip

{\it Proof of Claim 2. }It suffices to show that $x_{0}\neq x$ for each $x\in X\setminus\{y\}$. Assume on the contrary that $x_{0}\in X\setminus\{y\}$. In order to obtain a contradiction, we divide the proof in three
cases.

\smallskip
{\bf Subcase 1}: For some $n, m\in\omega$, we have $x_{0}=y(n, m)$.

\smallskip
Clearly, ones can choose $f\in ^\omega\omega$ such that $y(n, m)\notin V_f$. Fix an arbitrary $g\in ^{\omega_1}\omega$, and then let $P: =\{U(g,f), U(g,f),... \}\in\mathscr{U}^{\omega}$. Then
$(x_{0}+W(P))\cap H\neq \emptyset$. Since $x_{0}=y(n, m)\notin V_f$, each element of $(x_{0}+W(P))\cap H$ has the form $x_{0}+x'-y'$, where $(x', y')\in U(g, f)$. Put $$B: =\{x'-y': x_{0}+x'-y'\in (x_0+W(P))\cap H, (x', y')\in U(g, f)\}.$$ Then $B\subset -x_{0}+H$, and it follows from $x_{0}\in\overline{H}$ that $0\in \overline{B}$. Moreover, it is obvious that $-x_{0}+H$ is sequentially closed in $ A_{4}(X)$. Hence the set $(-x_{0}+H)\cap A_{2}(X)$ is sequentially closed in $A_{2}(X)$, then it follows that $(H-x_{0})\cap A_{2}(X)$ is closed since $A_{2}(X)$ is a sequential space. However, since $B\subset (-x_{0}+H)\cap A_{2}(X)$, we have $$0\in \overline{(-x_{0}+H)\cap A_{2}(X)}=(-x_{0}+H)\cap A_{2}(X)$$ in $A_{2}(X)$, which shows that $x_{0}\in H,$ this is a contradiction.

\smallskip
{\bf  Subcase 2}: For some $n\in\omega$ and $\beta\in\omega_1$, we have $x_{0}=x(n, \beta)$.

\smallskip
Fix arbitrary $f\in$ $^{\omega}\omega$ and $g\in$ $^{\omega_1}\omega$ with $g(\beta)>n$. Then $(x_{0}+W(P))\cap H\neq \emptyset$, where $P=(U(g, f))\in\mathscr{U}^{\omega}$. Obviously, each element of $(x_{0}+W(P))\cap H$ has the form $x_{0}+x_k'-y_k'\in H$, where $(x_k', y_k')\in U(g, f)$. By a proof analogous to Subcase 1, we can see that $x_{0}\in H$, which is a contradiction.

\smallskip
{\bf  Subcase 3}: For some $\beta\in\omega_1$, we have $x_{0}=x_\beta$.

Obviously, we have $x_{0}\in \overline{(x_{0}+W(P))\cap H}$ for any $P\in \mathscr{U}^\omega$. Fix $f\in$ $^\omega\omega$ and $g\in$ $^{\omega_1}\omega$. Let $$A: =(x_{0}+W(P))\cap H,$$ $$C: =\{x_{0}+x'-y': (x', y')\in U(g, f), x'\neq x_\beta, y'\neq x_\beta, x'\neq y'\}\cap H$$ and
$$D: =\{x'+x''-y'': (x'', y'')\in U(g, f), x'\in V_{g(\beta), \beta}\setminus\{x_{\beta}\}, x'\neq y''\}\cap H.$$ Clearly, we have $A=C\cup D$. If $x_{0}\in \overline{C}$, then, by a proof analogous to Subcase 1, we can see that $x_{0}\in H$, which is a contradiction. Hence we assume $x_{0}\in \overline{D}$.

Let $$B_1: =H\cap A_3(C_{\beta}\bigoplus S_\omega, S_\omega\bigoplus M_3)$$ and $$B_2: =H\cap A_3(M_3, S_\omega\bigoplus M_3),$$ where $C_{\beta}: =\{x_\beta\}\cup \{x(n, \beta): n\in \mathbb{N}\}$. By Lemma ~\ref{l-sequential}, the sets $B_1$ and $B_2$ are sequentially closed subsets of the spaces $A_3(C_{\beta}\bigoplus S_\omega, S_\omega\bigoplus M_3)$ and $A_3(M_3, S_\omega\bigoplus M_3)$ respectively, therefore, closed in $A(X)$. Then $B_1\cup B_2$ is closed in $A(X)$, hence we have $x_{0}\in \overline{D}\subset B_1\cup B_2\subset H$, which is a contradiction.

\smallskip
Therefore, we have $x_{0}=y$. The proof of Claim 2 is completed.

\smallskip
{\bf Claim 3:} The point $y\notin \overline{H}$ in $A(X)$.
\smallskip

{\it proof of Claim 3.} Obviously, it suffices to show that there exist two mappings $h\in ^\omega\omega$ and $g\in ^{\omega_1}\omega$ such that $(y+W(P))\cap H=\emptyset$, where $P=(U(g, h), U(g, h), ...)$. Since $y\not\in H$, it follows from Lemma~\ref{l-4} that we have the following statements:

\smallskip
($a$) there exists $h_1\in ^\omega\omega$ such that $\{y+x'-y': (x', y')\in V_{h_1}\times V_{h_1}\}\cap H=\emptyset$;
\smallskip

($b$) there exists $g_1\in ^{\omega_1}\omega$ such that $$\{y+x'-y': (x', y')\in V_{g_{1}(\alpha), \alpha}\times V_{g_{1}(\alpha), \alpha}, \alpha<\omega_1\}\cap H=\emptyset;$$

\smallskip
($c$) there exist $h_2, h_3\in ^\omega\omega$ such that $$\{y+x'-y'+x''-y'': (x', y')\in V_{h_2}\times V_{h_2}, (x'', y'')\in V_{h_3}\times V_{h_3}\}\cap H=\emptyset.$$

Therefore, without loss of generality, we may assume that $$H\subset \{y(n, m)+x_{1}-y_{1}: n, m\in \mathbb{N}, x_{1}, y_{1}\in C_{\alpha}, \alpha<\omega_1\}.$$

 Fix an arbitrary $\alpha<\omega_1$. Since $y\not\in H$, there exist a $f_\alpha\in$ $^\omega\omega$ and $n(\alpha)\in\mathbb{N}$ such that $$\{y(n, m)+x_{1}-y_{1}: n, m\in \mathbb{N}, m\geq f_\alpha(n), x_{1}, y_{1}\in V_{n(\alpha), \alpha}\}\cap H=\emptyset.\ \ \ \    (\star)$$
By the assumption of $\flat>\omega_1$, there is a $f\in ^\omega\omega$ such that $f_\alpha\leq^{*}f$ for each $\alpha<\omega_1$. For each $\alpha<\omega_1$, let $k(\alpha)$ be the smallest natural number such that $f_\alpha(k)\leq f(k)$ whenever $k\geq k(\alpha)$. We claim that, for each $\alpha<\omega_1$, there exists $n^{\prime}(\alpha)\in \mathbb{N}$ with $n^{\prime}(\alpha)\geq n(\alpha)$ such that $$\{y(j, m)+x_{1}-y_{1}: j<k(\alpha), m\in \mathbb{N}, x_{1}, y_{1}\in V_{n^{\prime}(\alpha), \alpha}\}\cap H=\emptyset.\ \ \ \ \ \ \ \    (\star\star)$$ Assume the converse. Then it is easy to see that there is a convergent sequence of $H$ with the limit point in $X$, which is a contradiction with the above assumption.
Let $g\in ^{\omega_1}\omega$ with $g(\alpha)=n'(\alpha)$ and $P=(U(g, f), U(g, f), ...)$. It follows from ($\star$) and ($\star\star$) ones can see $$\{y(n, m)+x_{1}-y_{1}: y(n, m)\in V_{f}, (x_{1}, y_{1})\in V_{g(\alpha), \alpha}\times V_{g(\alpha), \alpha}, \alpha<\omega_1, n, m\in\mathbb{N}\}\cap H=\emptyset.$$ Then we have $(y+W(P))\cap H=\emptyset$. The proof of Claim 3 is completed.

By Claims 2 and 3, we obtain a contradiction. Hence $A_{3}(X)$ is not a sequential space.
\end{proof}

\end{document}